\documentclass[3p,12pt]{mystyle}

\usepackage[T1]{fontenc}
\usepackage[latin1]{inputenc}

\usepackage{amsmath,amssymb,amsthm}

\usepackage{amsfonts}
\usepackage{wasysym}

\usepackage{epsfig}
\usepackage{epic}
\usepackage{eepic}
\usepackage{graphics}
\usepackage{graphicx}
\usepackage{color}
\usepackage{longtable}
\usepackage{subfigure}

\newcommand{\Dx}[1]{\mathrm{d}#1}

\newcommand{\ds}{\displaystyle}

\newcommand{\supp}{\operatorname{supp}}

\newcommand{\Nn}{{\mathbb N}}

\newcommand{\Rr}{{\mathbb R}}

\newcommand{\Zz}{{\mathbb Z}}

\newcommand{\Jj}{\mathbf{J}}

\newcommand{\muk}{\mu^{(k)}}
\newcommand{\munk}{\mu^{(n-k)}}
\newcommand{\mus}{\mu_{\Sigma}}
\newcommand{\muklambda}{\mu^{(k,\lambda)}}

\newcommand{\sigone}{\tau}
\newcommand{\sigtwo}{\sigma}
\newcommand{\sigN}{N}

\DeclareMathOperator*{\limweak}{w-lim}

\newtheorem*{theorem*}{Theorem}
\newtheorem{theorem}{Theorem}
\newtheorem{lemma}[theorem]{Lemma}

\newtheorem{corollary}[theorem]{Corollary}
\newdefinition{remark}{Remark}
\newdefinition{example}{Example}
\newdefinition{definition}{Definition}
\newdefinition{assumption}{Assumption}

\usepackage{tikz}
\usetikzlibrary{calc,fadings,decorations.pathreplacing}
\usetikzlibrary{matrix,arrows,positioning,fit,backgrounds} 
\usetikzlibrary{shapes,snakes}

\begin{document}

\begin{frontmatter}
\title{Weak limits for weighted means of orthogonal polynomials}

\author{Wolfgang Erb}
\ead{erb@math.hawaii.edu}

\address{ University of Hawai'i at M\=anoa, Department of Mathematics \\
          2565 McCarthy Mall, Keller Hall 401A, Honolulu, HI, 96822}
	
\date{\today}

\begin{abstract}
This article is a first attempt to obtain weak limit formulas for weighted means of orthogonal polynomials. For this,
we introduce a new mean Nevai class that guarantees the existence of an equilibrium measure for the limit of the means. We show that for a family of 
measures in this mean Nevai class also the means of the Christoffel-Darboux kernels 
and the asymptotic distribution of the roots converge weakly to the same equilibrium measure. As a main example, 
we study the mean Nevai classes in which the equilibrium measure is the orthogonality measure of the ultraspherical polynomials. The 
respective weak limit formula can be regarded as an asymptotic weak addition formula for the corresponding class of measures. 
\end{abstract}


\end{frontmatter}

\section{Introduction}

A major result for weak limits of measures related to orthogonal polynomials on the real line is the following statement given in \cite[Section 4, Theorem 14]{Nevai}:

\begin{theorem}[Nevai 1979, \cite{Nevai}] \label{thm-Nevai}
Let $\mu$ be a measure in the Nevai class $M(a,b)$ with $b > 0$. If $f$ is $\mu$-measurable, bounded on $\supp(\mu)$ and Riemann integrable on $[-2b+a,a+2b]$, then
\begin{equation} \label{eq:limitNevai} \lim_{n \to \infty} \int_\Rr f(x) p_n^2(x) \Dx{\mu} = \frac{1}{\pi} \int_{a-2b}^{a+2b} f(x) \frac{\Dx{x}}{\sqrt{4 b^2-(x-a)^2}}.\end{equation}
\end{theorem}
The polynomials $p_n$ of degree $n \in \Nn_0$ are the orthonormal polynomial with respect to the measure $\mu$. They satisfy a three-term recurrence relation  
$$x p_n(x) = b_{n+1} p_{n+1}(x) + a_n p_n(x) + b_n p_{n-1}$$ with coefficients $a_n \in \Rr$ and $b_{n+1}>0$, $n \in \Nn_0$. The \emph{Nevai class $M(a,b)$} in Theorem \ref{thm-Nevai} is
the set of all measures $\mu$ such that $\lim_{n \to \infty} a_n = a$ and $\lim_{n \to \infty} b_n = b$. In \cite{Simon2004}, also a converse statement is shown, namely that 
if $\supp(\mu)$ is bounded and \eqref{eq:limitNevai} holds true for all continuous functions $f$ on $\supp(\mu)$ then $\mu$ is in the Nevai class $M(a,b)$. Further,
in \cite{Simon2004} the existence of this statement is extended to a larger class of measures in which the recurrence coefficients satisfy the weaker condition
$\lim_{n \to \infty} b_{2n} = b'$, $\lim_{n \to \infty} b_{2n+1} = b''$ and $\lim_{n \to \infty} a_n = a$. 
A similar weak limit result holds true for orthogonal polynomials on the unit circle. The corresponding statement in this setting is Khrushchev's Theorem: the weak convergence 
of $|p_n(e^{i \theta})|^2 \Dx{\mu}$ towards the uniform measure $ \frac{\Dx{\theta}}{2\pi}$ on the unit circle is equivalent to the fact that the Verblunsky coefficients satisfy
the M{\'a}t{\'e}-Nevai condition \cite[Theorem 9.3.1]{Simon2005}. 

If the assumptions of Theorem \ref{thm-Nevai} are satisfied, two further related weak limits can be derived
from \eqref{eq:limitNevai} (see \cite[Section 5, Lemma 1 and Theorem 3]{Nevai}), namely
\begin{align} \label{eq:CDlimitNevai} 
\lim_{n \to \infty} \int_\Rr f(x)  \frac{1}{n+1} \sum_{k = 0}^n p_k^2(x) \Dx{\mu} &= \frac{1}{\pi} \int_{a-2b}^{a+2b}  \frac{f(x) \, \Dx{x}}{\sqrt{4 b^2-(x-a)^2}}, \\
\lim_{n \to \infty} \frac{1}{n+1} \sum_{k = 1}^{n+1} f(x_{n+1,k}) &=  \frac{1}{\pi} \int_{a-2b}^{a+2b}  \frac{f(x) \, \Dx{x}}{\sqrt{4 b^2-(x-a)^2}}. \label{eq:discretelimitNevai}
\end{align}
The limit \eqref{eq:CDlimitNevai} is usually referred to as weak limit of the Christoffel-Darboux kernel (or shortly CD kernel) $K(x,x) = \sum_{k=0}^n p_k^2(x)$. 
The values $x_{n+1,1} < x_{n+1,2}< \cdots < x_{n+1,n+1}$ denote the $n+1$ roots of the polynomial $p_{n+1}$ and the 
weak limit in \eqref{eq:discretelimitNevai} therefore gives the asymptotic distribution of the roots of the orthogonal polynomials $p_{n+1}$. 
The two sequences of measures in \eqref{eq:CDlimitNevai} and \eqref{eq:discretelimitNevai} are intimately related. In 
\cite{Simon2009} it is shown that the convergence of a subsequence of one of the two sequences implies the convergence of the corresponding other.
Variants and generalizations of the limits \eqref{eq:limitNevai}, \eqref{eq:CDlimitNevai} and \eqref{eq:discretelimitNevai} have been intensively studied for 
different families of orthogonal polynomials,
among others, for orthogonal polynomials on the unit circle and for classes of measures with asymptotically periodic recurrence coefficients. A general overview can be found 
in the book \cite{Simon2011}. Specific variants are, for instance, discussed in \cite{Christiansen2009,DamanikKillipSimon2014,MateNevaiTotik1987,MhaskarSaff1990,Nevai1979,VanAssche}. 

In this paper we want to derive and investigate analogs of the weak limits \eqref{eq:limitNevai}, \eqref{eq:CDlimitNevai} and \eqref{eq:discretelimitNevai} for weighted means of differing orthogonality measures. The motivation to study such mean weak limits originates in two works \cite{Erb2013} and \cite{ErbMathias2015} in which a Landau-Pollak-Slepian type
space-frequency analysis was studied for spaces of orthogonal polynomials. In the one-dimensional case given in \cite{Erb2013}, the roots of orthogonal polynomials
were used to describe the spatial position of localized basis functions. The corresponding asymptotic distribution of the roots is given by the arcsine distribution \eqref{eq:discretelimitNevai}. 
In the case of the unit sphere a similar description was derived in \cite{ErbMathias2015}. This description however included ultraspherical and co-recursive 
ultraspherical polynomials with a differing parameter. In this setting, 
the mean asymptotic distribution of the roots of the polynomials turned out to be a uniform distribution on $[-1,1]$. 
Compared to the arcsine distribution given in \eqref{eq:discretelimitNevai} this result came as a surprise. Our aim here is to obtain a better understanding of the differences
in the two settings. 

Our first goal is to introduce a new Nevai class for families of orthogonality measures that guarantees the existence of weak limits for means of measures. 
For this we will shortly recapitulate some facts about regular summation methods. The extension of Theorem \ref{thm-Nevai} 
to weighted means of orthogonal polynomials is formulated in Theorem \ref{theorem-Nevaiweaklimit}. The analogs of the formulas \eqref{eq:CDlimitNevai} 
for the Christoffel-Darboux kernel and \eqref{eq:discretelimitNevai} for the mean asymptotic distribution of the roots
are provided in Theorem \ref{theorem-NevaiweaklimitCD} and Theorem \ref{theorem-Nevaiweaklimitpointmeasure}, respectively.
Finally, we will investigate some particular summation methods and the corresponding mean Nevai classes 
in which the equilibrium measure for the weak limit is precisely the orthogonality measure of the ultraspherical polynomials. The so
obtained limit formulas can be regarded as asymptotic weak addition formulas for the underlying Nevai class.

\section{Preliminaries}

We consider a family of non-negative measures $\muk$, $k \in \Nn_0$ on $\Rr$ supported on a bounded subinterval of $\Rr$. By 
$p_l^{(k)}(x)$, we denote the corresponding family of orthogonal polynomials of degree $l$. The polynomials $p_l^{(k)}(x)$ are normalized 
such that the leading coefficient is positive and that $p_l^{(k)}(x)$ are orthonormal on $\Rr$ with respect to the inner product
	\begin{equation*} 
	\label{eq:innerproduct-ultra}
	\langle f,g \rangle^{(k)} := \int_{\Rr} f(x) \overline{g( x)} \Dx{\muk}.
	\end{equation*}
It is well known that the polynomials $p_l^{(k)}$ are an orthonormal basis of the Hilbert space $(L^2(\Rr,\muk), \langle \cdot, \cdot \rangle^{(k)})$.
For a general overview on orthogonal polynomials and a multitude of their properties, we refer to the monographs \cite{Chihara,Gautschi,Ismail,Szegoe}. For this article, the three-term recurrence relation
of the orthogonal polynomials is of major importance: setting $p_{-1}(x) = 0$ and $p_0 = \frac{1}{\muk(\Rr)}$, we have the relation
\begin{equation} \label{eq:threetermrecurrencegeneral}
x p_l^{(k)}(x) = b_{l+1}^{(k)} p_{l+1}^{(k)}(x) + a_l^{(k)} p_l^{(k)}(x) + b_l^{(k)} p_{l-1}^{(k)}
\end{equation}
for $l \in \Nn_0$ with coefficients $a_l^{(k)} \in \Rr$ and $b_l^{(k)} > 0$. Introducing the Jacobi matrices
	\begin{equation*}
	\label{eq:jacobi}
		\Jj_l^{(k)} := 	
		\begin{pmatrix}
			a_0^{(k)} 	& b_{1}^{(k)}	 & 0      	& 0 	    	& \cdots  	& 0 \\
			b_{1}^{(k)} 	& a_1^{(k)} 	 & b_{2}^{(k)}  & 0 	    	& \cdots  	& 0 \\
			0 		& b_{2}^{(k)}    & a_2^{(k)}	& b_{3}^{(k)}   & \ddots  	& \vdots \\
			\vdots 		& \ddots 	 & \ddots 	& \ddots  	& \ddots	& 0 \\
			0 		& \cdots 	 & 0      	& b_{l-1}^{(k)} & a_{l-1}^{(k)} & b_{l}^{(k)} \\
			0 		& \cdots 	 & \cdots 	& 0 	    	& b_{l}^{(k)}   & a_l^{(k)}
		\end{pmatrix},
	\end{equation*}
we further have the representation 
\begin{align*}
	p_{l+1}^{(k)}(x) &= \frac{1}{\muk(\Rr)}\frac{1}{ b_{1}^{(k)} \cdots b_{l+1}^{(k)}}\det \left(x \mathbf{1}_{l+1} - \Jj_{l}^{(k)} \right). \label{eq:expressionbyjacobianassociated}
\end{align*}
 From this representation it is obvious that the roots $x_{l+1,j}^{(k)}$, $1 \leq j \leq l+1$, of $p_{l+1}^{(k)}(x)$ correspond to the $l+1$ eigenvalues of the 
 symmetric matrix $\Jj_l^{(k)}$. To simplify our calculations, we additionally set $a_{-l}^{(k)} = 0$ if $l \in \Nn$ and $b_{-l}^{(k)} = 0$ for $l \in \Nn_0$.

\section{Weak limits for weighted sums of orthogonal polynomials}

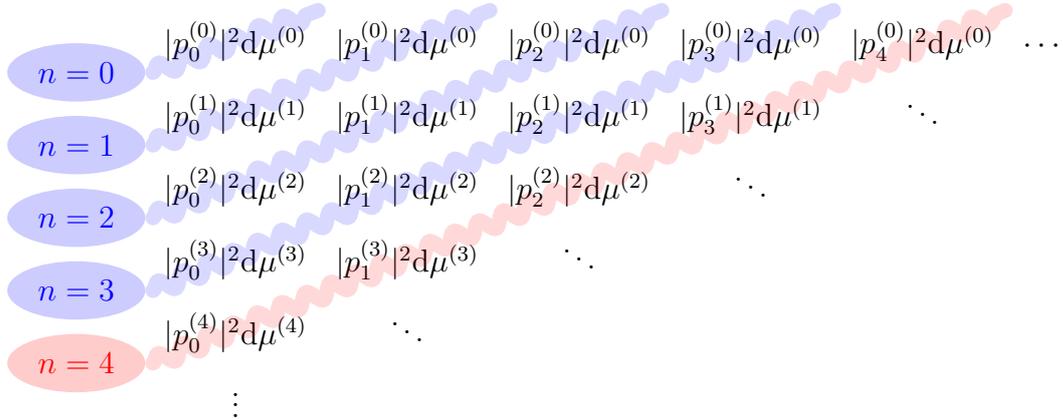
\begin{figure} \centering
\begin{tikzpicture}[baseline=(A.center)]
  \tikzset{BarreStyle/.style =   {opacity=.15,line width=2 mm,line cap=round,color=#1}}
  \tikzset{SignePlus/.style =   {left,,opacity=1,ellipse,fill=#1!20}}
\matrix (A) [matrix of math nodes,,,column sep=0.1cm, row sep=0.05cm] 
{ {|p_0^{(0)}|^2} \Dx{\mu^{(0)}}	& {|p_1^{(0)}|^2} \Dx{\mu^{(0)}}	& {|p_2^{(0)}|^2} \Dx{\mu^{(0)}}	& {|p_3^{(0)}|^2} \Dx{\mu^{(0)}}	& {|p_4^{(0)}|^2} \Dx{\mu^{(0)}}		& \cdots & \\
  {|p_0^{(1)}|^2} \Dx{\mu^{(1)}}	& {|p_1^{(1)}|^2} \Dx{\mu^{(1)}}	& {|p_2^{(1)}|^2} \Dx{\mu^{(1)}}	& {|p_3^{(1)}|^2} \Dx{\mu^{(1)}}	& \ddots	&  & \\
  {|p_0^{(2)}|^2} \Dx{\mu^{(2)}}	& {|p_1^{(2)}|^2} \Dx{\mu^{(2)}}	& {|p_2^{(2)}|^2} \Dx{\mu^{(2)}} 	& \ddots	& 	&  & \\  
  {|p_0^{(3)}|^2} \Dx{\mu^{(3)}}	& {|p_1^{(3)}|^2} \Dx{\mu^{(3)}}	& \ddots	& 	&        &  & \\
  {|p_0^{(4)}|^2} \Dx{\mu^{(4)}}	& \ddots	& 	&  	&        &  & \\
  \vdots		& 		& 			& 			       		&  & \\
};
 \draw [BarreStyle=blue,decorate,decoration=snake]  (A-1-1.south west) node[SignePlus=blue] {$n=0$} to (A-1-1.north east);
 \draw [BarreStyle=blue,decorate,decoration=snake]  (A-2-1.south west) node[SignePlus=blue] {$n=1$} to (A-1-2.north east);
 \draw [BarreStyle=blue,decorate,decoration=snake]  (A-3-1.south west) node[SignePlus=blue] {$n=2$} to (A-1-3.north east);
 \draw [BarreStyle=blue,decorate,decoration=snake]  (A-4-1.south west) node[SignePlus=blue] {$n=3$} to (A-1-4.north east);
 \draw [BarreStyle=red,decorate,decoration=snake]  (A-5-1.south west) node[SignePlus=red] {$n=4$} to (A-1-5.north east);
\end{tikzpicture}
\caption{Graphical illustration of the summation methods. Summing over the red curve gives the measure $ \bar{\mu}_4$ defined in \eqref{eq:summethod1}, summing over
all blue curves generates the measure $\lambda_3$ introduced in \eqref{eq:summethod2}. The sum of all red and blue elements gives $\lambda_4$.}
\label{fig:1}
\end{figure}

For a sequence $y = (y_k)_{k \in \Nn_0}$, we consider summation methods $S = (S_n)_{n \in \Nn_0}$ given by
\[ S_n ( y ) = \sum_{k=0}^n \sigma_{n,k} y_k.\]

\begin{assumption}[Regularity of $S$] \label{assumption1}
We assume that the weights $\sigma_{n,k}$ of the summation method $S$ satisfy the following three conditions:
\begin{enumerate}[(i)]
 \item $\sigma_{n,k} \geq 0$, for all $0 \leq k \leq n$, $n \in \Nn_0$,  
 \item $\sum_{k=0}^{n} \sigma_{n,k} = 1$,
 \item $\lim_{n \to \infty} \sigma_{n,k} = 0$ for $k \in \Nn_0$.
\end{enumerate}
These three conditions imply that, according to the Silverman-Toeplitz theorem, for every convergent sequence $y$ the 
sequence $(S_n y)_{n \in \Nn_0}$ converges to the same limit. In this work, we call the summation method $S$ regular if the three conditions (i), (ii) and (iii) are
satisfied. 
\end{assumption}

Note that in the literature the non-negativity of the summation weights is usually not demanded for regularity. If negative weights $\sigma_{n,k}$ are allowed, 
Assumption \ref{assumption1} on regularity can be weakened by postulating (ii), (iii) together with $\sum_{k=0}^{n} |\sigma_{n,k}| \leq M$ for some positive constant $M$. 
In this work, we will only consider non-negative weights and therefore use the conditions in Assumption \ref{assumption1}.
For a broader overview to summation methods we refer to \cite{BoosCass}. 

Based on a regular summation method $S$, we consider now for $n \geq 0$ the following mean measures. The single measures involved in the summation are illustrated in Figure \ref{fig:1}.
\begin{equation} \label{eq:summethod1}
 \Dx{\bar{\mu}_n} = \sum_{k=0}^n \sigma_{n,k} \left( p_{k}^{(n-k)}(x)\right)^2 \Dx{\munk}.
\end{equation}

\begin{definition} \label{def:meannevai}
For a regular summation method $S$, we say that the family of measures $\muk$, $k \in \Nn_0$, is in the mean Nevai class $M^{(S)}(\Sigma_{l_a,l_b})$
if the following three conditions are satisfied for the coefficients $a_n^{(k)}$, $b_n^{(k)}$ in the three-term recurrence relation \eqref{eq:threetermrecurrencegeneral}:
\begin{enumerate}[(i)]
 \item $\ds \sup_{n,k \in \Nn_0} |a_n^{(k)}| = A < \infty$, $\ds \sup_{n,k \in \Nn_0} |b_n^{(k)}| = B < \infty$.
 \item $\ds \lim_{n \to \infty} \sum_{k = 0}^n \sigma_{n,k} |a_{k+l}^{(n-k)} - a_{k}^{(n-k)}| = 0$, $\ds \lim_{n \to \infty} \sum_{k = 0}^n \sigma_{n,k} |b_{k+l}^{(n-k)} - b_{k}^{(n-k)}| = 0$, for
 all $l \in \Nn$.
 \item $\ds \lim_{n \to \infty} \sum_{k = 0}^n \sigma_{n,k} (a_{k}^{(n-k)})^{l_a} (b_{k}^{(n-k)})^{l_b} = \Sigma_{l_a,l_b} < \infty$ for all $l_a,l_b \in \Nn_0$. 
\end{enumerate}
\end{definition}

\begin{theorem} \label{theorem-Nevaiweaklimit}
Suppose that the family $\muk$, $k \in \Nn_0$, is in the mean Nevai class $M^{(S)}(\Sigma_{l_a,l_b})$. Then, 
the sequence $\bar{\mu}_n$ converges weakly to an equilibrium measure $\mus$ and for every continuous $f$ we have
 \[ \lim_{n \to \infty} \sum_{k=0}^n \sigma_{n,k}\int_{\Rr} f(x) (p_{k}^{(n-k)}(x))^2 \Dx{\mu^{(n-k)}} = \int_{\Rr} f(x) d \mus.\]
 The equilibrium measure $\mus$ is determined by the numbers $\Sigma_{l_a,l_b}$, $l_a\in \Nn_0$, $l_b \in 2 \Nn_0$.  
\end{theorem}

In order to prove Theorem \ref{theorem-Nevaiweaklimit}, we can follow similar argumentation lines as given in \cite[Section 4.2, Theorem 14]{Nevai} or 
in \cite[Theorem 3.1]{Simon2004}. In the following, we will stay closer to the notion used in \cite{Simon2004} and introduce some additional terminology.  

A path $\rho$ of length $l$ is a sequence $(\rho_j)_{j=0}^{l} \in \Zz^{l+1}$ of $l+1$ integers such that $|\rho_j - \rho_{j-1}| \leq 1$. For a path $\rho$, we define the weights
\[ W^{(k)}(\rho) = \prod_{j=0}^{l-1} w^{(k)}(\rho_j,\rho_{j+1})\]
where
\[ w^{(k)}(\rho_j,\rho_{j+1}) = \left\{ \begin{array}{ll}
                                       a_{m}^{(k)} & \text{if} \; \rho_{j+1} = \rho_{j} = m,   \\ 
                                       b_{m+1}^{(k)} & \text{if} \; \rho_{j+1} = \rho_{j}+1 = m + 1, \\ 
                                       b_{m}^{(k)} & \text{if} \;\rho_{j+1} = \rho_{j}-1 = m-1.
                                       \end{array} \right.\]
Remind that for negative $m$ we set the coefficients $a_{m}^{(k)}$ and $b_{m}^{(k)}$ equal to zero, as well as the coefficient $b_{0}^{(k)}$. 
For a path $\rho$ and $m\in \Zz$, we further define the shifted path $T_m \rho$ by
\[T_m \rho = (\rho_j + m)_{j=0}^{l}.\]

By using the three-term recurrence relation \eqref{eq:threetermrecurrencegeneral} it is straightforward to see that 
\begin{equation} \label{eq:helpmeBarry1} \int_\Rr x^l \left( p_{m}^{(k)}(x)\right)^2 \Dx{\muk} = \langle x^j p_{m}^{(k)},p_{m}^{(k)} \rangle^{(k)} = \sum_{\rho \in Q_{l}} W^{(k)}(T_m \rho) \end{equation}
holds, where $Q_{l}$ denotes the set of all paths of length $l$ with $\rho_0 = \rho_l = 0$. The detailed elaboration for the derivation of \eqref{eq:helpmeBarry1} is given
in \cite[Proposition 3.3]{Simon2004}.
The identity \eqref{eq:helpmeBarry1} immediately implies the formula
\begin{equation} \label{eq-help1} \int_{\Rr} x^l \Dx{\bar{\mu}_n} = \sum_{\rho \in Q_{l}} \sum_{k=0}^n \sigma_{n,k} W^{(n-k)}(T_{k} \rho) \end{equation}
for the mean measures $\muk$. It allows us to prove Theorem \ref{theorem-Nevaiweaklimit}.

\begin{proof}
 The assumption $(i)$ in Theorem \ref{theorem-Nevaiweaklimit} implies that the measures $\bar{\mu}_n$, $n \in \Nn_0$, are all compactly supported in $[-2 B-A,A+2B]$. 
 Thus, weak convergence of the measures $\bar{\mu}_n$ is equivalent to the convergence of the moments
 \[\int_{\Rr} x^l \Dx{\bar{\mu}_n}.\]
 By \eqref{eq-help1}, we only have to show that for every path $\rho \in Q_l$ the sums 
\[ \sum_{k=0}^n \sigma_{n,k} W^{(n-k)}(T_{k} \rho)\]
are converging. For this, we refine our look on the paths $\rho \in Q_l$. For $\rho \in Q_l$ with length $l$ we set $l_a = \# \{j:\; \rho_{j+1} = \rho_{j} \}$ and 
$l_b = \# \{j:\; \rho_{j+1} \neq \rho_{j} \}$ such that $l = l_a + l_b$ holds. Then, the sum above can be written as
\[ \sum_{k=0}^n \sigma_{n,k} \prod_{j=0}^{l_a-1} a_{k+r_{j,a}}^{(n-k)}\prod_{j=0}^{l_b-1} b_{k+r_{j,b}}^{(n-k)}, \]
where $r_{j,a}$ and $r_{j,b}$ are integers between $-l/2$ and $l/2$ that depend on the chosen path $\rho \in Q_l$. Further, since $\rho \in Q_l$, $l_b$ must be an even integer. Because
of the assumptions (i), (ii) and (iii) in the Definition \ref{def:meannevai} of the Nevai class $M^{(S)}(\Sigma_{l_a,l_b})$, we obtain as $n \to \infty$ the identity 
\[ \lim_{n \to \infty} \sum_{k=0}^n \sigma_{n,k} \prod_{j=0}^{l_a-1} a_{k+r_{j,1}}^{(n-k)}\prod_{j=0}^{l_b-1} b_{k+r_{j,2}}^{(n-k)}
 = \lim_{n \to \infty} \sum_{k=0}^n \sigma_{n,k} (a_{k}^{(n-k)})^{l_a} (b_{k}^{(n-k)})^{l_b} = \Sigma_{l_a,l_b}. \]
In this formula, the condition (iii) guarantees the existence of the limit in the second equality, whereas the conditions (i) and (ii) guarantee that the first equality holds true.  
Therefore, starting from the identity \eqref{eq-help1} we obtain
\begin{equation*} \lim_{n\to \infty} \int_{\Rr} x^l \Dx{\bar{\mu}_n} = \sum_{\rho \in Q_{l}} \Sigma_{l_a,l_b},  \end{equation*} 
and thus the statement of the theorem. In particular, the last identity implies that the equilibrium measure $\mus$ solely depends on the values $\Sigma_{l_a,l_b}$, $l_a \in \Nn$, $l_b \in 2 \Nn_0$. 
\end{proof}

\begin{remark}
The assumption $(i)$ in Definition \ref{def:meannevai} implies that the measures $\muk$, the means $\bar{\mu}_n$ and the equilibrium measure $\mus$ are all compactly supported in $[-2 B-A,A+2B]$.
\end{remark}

We give a first example of such a mean weak limit. In view of Theorem \ref{thm-Nevai} for measures in the Nevai class $M(a,b)$ the outcome is not yet surprising.  

\begin{definition} \label{def:uniformnevai}
We say that the family of measures $\muk$, $k\in \Nn_0$, is in the uniform Nevai class $M^{(U)}(a,b)$ if the  
coefficients $a_n^{(k)}$, $b_n^{(k)}$ in the three-term recurrence relation \eqref{eq:threetermrecurrencegeneral} are uniformly in 
the same Nevai class $M(a,b)$, i.e.
\begin{equation} \label{eq:uniformnevai} \lim_{n \to \infty} \sup_{k \in \Nn} |a_n^{(k)} - a| = 0, \quad \lim_{n \to \infty} \sup_{k \in \Nn} |b_n^{(k)} - b| = 0. \end{equation}
\end{definition}

We denote the characteristic function of an interval $I \subset \Rr$ by $\chi_I$. 

\begin{corollary} \label{cor:uniformnevai}
Assume that the family $\muk$, $k\in \Nn_0$, is in the uniform Nevai class $M^{(U)}(a,b)$. Then, for every regular summation method $S$
the family $\muk$ is in the mean Nevai class 
$M^{(S)}(\Sigma_{l_a,l_b})$ with the values $\Sigma_{l_a,l_b}$ given by $\Sigma_{l_a,l_b} = a^{l_a} b^{l_b}$. In this case, the equilibrium measure $\mus$ is explicitly given by
\[ \Dx{\mus} = \frac{1}{\pi} \frac{\chi_{[a-2b,a+2b]}(x)}{\sqrt{4b^2 - (x-a)^2}}\Dx{x}. \]
In particular, the measure $\mus$ depends only on the limits $a$ and $b$ and is independent of the summation method $S$. 
\end{corollary}

\begin{proof}
It is straightforward to check that the uniformity of the limits in \eqref{eq:uniformnevai} together with the fact that the summation method $S$ is regular according to 
Assumption \ref{assumption1} implies all three conditions in Definition \ref{def:meannevai}, and also that
$\Sigma_{l_a,l_b} = a^{l_a} b^{l_b}$. By Theorem \ref{theorem-Nevaiweaklimit}, we therefore get an equilibrium measure $\mus$ which only 
depends on the limits $a$ and $b$ of the class. Further, to get the explicit form of the equilibrium measure $\mus$, it is sufficient to derive it for one particular summation method and 
for one particular family of measures in $M^{(U)}(a,b)$. For this, we consider a measure $\mu$ in the Nevai class $M(a,b)$ and set $\muk = \mu$ for all $k \in \Nn_0$. 
Then, the family $\muk$ is in $M^{(U)}(a,b)$. As a summation method $S$ we consider the identity scheme given by $\sigma_{n,k} = \delta_{n k}$, where $\delta_{n k}$ denotes 
the usual Kronecker delta. Obviously, this summation method is regular. For this construction, we explicitly get
\begin{align*}
\lim_{n \to \infty} \sum_{k=0}^n \sigma_{n,k}\int_{\Rr} f(x) (p_{k}^{(n-k)}(x))^2 \Dx{\mu^{(n-k)}} &= 
\lim_{n \to \infty} \int_{\Rr} f(x) (p_{n}(x))^2 \Dx{\mu} \\&= \frac{1}{\pi} \int_{a-2b}^{a+2b} \frac{f(x)}{\sqrt{4b^2 - (x-a)^2}}\Dx{x},
\end{align*}
where the last equality follows by Theorem \ref{thm-Nevai}. 
\end{proof}

\section{Weak limits for weighted sums of Christoffel-Darboux kernels}

We are now interested in weak limits related to the family $K_n^{(k)}(x,y)$, $k \in \Nn_0$, of Christoffel-Darboux kernels. They are defined as
\[K_n^{(k)}(x,y) = \sum_{l=0}^n p_l^{(k)}(x) p_l^{(k)}(y), \qquad n \in \Nn_0. \]
By considering averages of the measures $K_k^{(n-k)}(x,x) \Dx{\munk}$, $k = 0, \ldots, n$, the weights $\sigma_{n,k}$ of the summation method should in principal depend 
on $n$ and on the parameter $n-k$ of the measure $\munk$. Therefore, in this section we restrict ourselves to the following class of summation methods $S$.

\begin{definition}
Let $(\sigtwo_n)_{n\in \Nn_0}$ and $(\sigone_n)_{n\in\Nn_0}$ be two non-negative sequences. We call the summation method $S$ a N\"orlund method if the weights $\sigma_{n,k}$ 
are given by 
\begin{equation} \label{eq:decompositionsummation}
\sigma_{n,k} = \sigone_n \sigtwo_{n-k}, \quad n \in \Nn_0,\; 0 \leq k \leq n.
\end{equation}
We call a N\"orlund method $S$ regular if the three conditions of Assumption \ref{assumption1} are satisfied. 
\end{definition}

\begin{remark}
For a regular N\"orlund method $S$ with a fixed sequence 
$(\sigtwo_n)_{n\in \Nn_0}$ the sequence $(\sigone_n)_{n\in\Nn_0}$ is uniquely determined by
\begin{equation} \label{eq-tauunique} \sigone_n = \left(\sum_{k=0}^n \sigtwo_{n-k}\right)^{-1} > 0.\end{equation}
The fact that $S$ is regular implies further that $\sigtwo_0>0$. Therefore, we can always normalize a regular N\"orlund method such that $\sigtwo_{0} = 1$, $\tau_{0} = 1$ and 
$0 <\sigone_n \leq 1$ for $n \geq 1$. For more properties on N\"orlund methods we refer to \cite[Section 3.3]{BoosCass}. 
\end{remark}

For a regular N\"orlund method $S$ based on the sequences 
$(\sigtwo_n)_{n\in \Nn_0}$ and $\sigone_n = (\sum_{k=0}^n \sigtwo_{n-k})^{-1}$ we can define the normalization constant
\[ \sigN_n = \sigone_n \sum_{k=0}^n \frac{1}{\sigone_k} \geq 1.\]
This gives us the possibility to introduce a new summation method $T$ based on the positive weights
\begin{equation} \label{eq:secondnorlund} \tau_{n,k} = \frac{\sigone_n}{\sigN_n} \frac{1}{\sigone_{k}}.\end{equation}

\begin{lemma} \label{lemma-riesz}
If $S$ is a regular N\"orlund method determined by a given sequence $(\sigtwo_n)_{n\in \Nn_0}$ in \eqref{eq:decompositionsummation},
then the summation method $T$ given by \eqref{eq:secondnorlund} is also regular. 
\end{lemma}

\begin{proof}
The first two conditions (i) and (ii) of Assumption \ref{assumption1} are automatically satisfied by the construction of the summation method $T$.
Therefore, we only have to show (iii), i.e. that for fixed $k$ the positive sequence $\tau_{n,k} = \frac{\sigone_n}{\sigN_n} \frac{1}{\sigone_{k}}$ is converging to zero as $n \to \infty$. 
We know by \eqref{eq-tauunique} that the sequence $(\sigone_n)$ is positive and 
monotonically decreasing and, thus, convergent. If $(\sigone_n) \to 0$ then also $\frac{\sigone_n}{\sigN_n}$ converges to $0$ as $n \to \infty$. If 
$(\sigone_n) \to c > 0$, then $\sigN_n \to \infty$ as $n \to \infty$. Thus, also in this second case the sequence $\frac{\sigone_n}{\sigN_n}$ tends to zero. 
\end{proof}

\begin{remark}
The summation method $T$ is a Riesz summation method according to the definition given in \cite[Definition 3.2.2]{BoosCass}. Since $(\sigone_n)$ is monotonically decreasing,
the maximum of $\tau_{n,k}$ for fixed $n$ is attained at $k = n$. In particular, we have the inequalities $0 < \tau_{n,k} \leq \frac{1}{N_n}$. Since the summation method $T$ is regular, 
we therefore get the following estimates for the normalization constant $N_n$:
\[ 1 \leq N_n \leq n+1. \]
\end{remark}

We investigate now the following means related to the kernels $K_k^{(n-k)}(x,y)$:
\begin{equation}
  \Dx{\lambda_n} = \frac{1}{N_n} \sum_{k=0}^n \sigma_{n,k} K_k^{(n-k)}(x,x) \Dx{\munk}. \label{eq:summethod2}
\end{equation}

\begin{theorem} \label{theorem-NevaiweaklimitCD} Let $S$ be a regular N\"orlund method and let $\muk$, $k\in \Nn_0$, be in the mean Nevai class $M^{(S)}(\Sigma_{l_a,l_b})$. Then, 
the sequence $\lambda_n$ converges weakly to the equilibrium measure $\mus$ given in Theorem \ref{theorem-Nevaiweaklimit}, i.e. 
\begin{equation*} \label{equation-weaklimits} 
\limweak_{n \to \infty} \lambda_{n} = \limweak_{n \to \infty} \bar{\mu}_n = \mu_{\Sigma}.
\end{equation*} 
In particular, for every continuous function $f$ we have
 \begin{align*} 
 \lim_{n \to \infty} \frac{1}{N_n} \sum_{k=0}^n \sigma_{n,k} \int_{\Rr} f(x) K_{k}^{(n-k)}(x,x) \Dx{\mu^{(n-k)}} = \int_{\Rr} f(x) d \mus.
 \end{align*}
\end{theorem}

\begin{proof}
Since $S$ is a regular N\"orlund method, we can decompose the weights $\sigma_{n,k}$ and rearrange the sums in the definition \eqref{eq:summethod2} of the measure $\lambda_n$.
In this way we get
\begin{align*}
\Dx{\lambda_n} &= \frac{1}{N_n} \sum_{k=0}^n \sigma_{n,k} K_k^{(n-k)}(x,x) \Dx{\munk} = 
\frac{\sigone_{n}}{N_n} \sum_{k=0}^n  \sigtwo_{n-k} \sum_{j=0}^{k} |p_j^{(n-k)}(x)|^2 \Dx{\munk} \\
&= \frac{\sigone_{n}}{N_n} \sum_{k=0}^n  \sum_{j=0}^{k} \sigtwo_{k-j}  |p_j^{(k-j)}(x)|^2 \Dx{\mu^{(k-j)}} 
= \frac{\sigone_{n}}{N_n} \sum_{k=0}^n  \frac{1}{\sigone_{k}}\sum_{j=0}^{k} \sigma_{k,j}  |p_j^{(k-j)}(x)|^2 \Dx{\mu^{(k-j)}}\\
&= \sum_{k=0}^n  \tau_{n,k} \, \Dx{\bar{\mu}_{k}}.
\end{align*}
By Theorem \ref{theorem-Nevaiweaklimit}, we know that $\limweak_{n \to \infty} \bar{\mu}_{n} = \mu_{\Sigma}$. By Lemma \ref{lemma-riesz}, the Riesz summation 
method $T$ is regular and, thus, preserves the weak limit $\mu_{\Sigma}$, i.e. 
$$\limweak_{n \to \infty} \lambda_{n} = \limweak_{n \to \infty} \sum_{k=0}^n  \tau_{n,k} \, \bar{\mu}_{k} = \limweak_{n \to \infty} \bar{\mu}_{n} = \mu_{\Sigma}.$$ 
\end{proof}

Now, let $x_{n,j}^{(k)}$ denote the $j$th smallest root of the orthogonal polynomial $p_n^{(k)}$ and $\delta_x$ be the Dirac point measure supported at $x \in \Rr$.
We study the limiting properties of the following family of discrete measures:
\begin{align}
   \nu_{n+1} &= \frac{1}{N_n} \sum_{k=0}^{n} \sigma_{n,k} \sum_{j=1}^{k+1} \delta_{x_{k+1,j}^{(n-k)}}.  \label{eq:summethod3}
\end{align}

\begin{theorem} \label{theorem-Nevaiweaklimitpointmeasure} Let $S$ be a regular N\"orlund method and assume that $N_n \to \infty$. 
Let the family $\muk$, $k\in \Nn_0$, be in the mean Nevai class $M^{(S)}(\Sigma_{l_a,l_b})$. Then, 
the sequence $\nu_{n+1}$ converges also weakly to the equilibrium measure $\mus$ given in Theorem \ref{theorem-Nevaiweaklimit}. In particular, we have 
\begin{equation*} \label{equation-weaklimits2} 
\limweak_{n \to \infty} \nu_{n+1} = \limweak_{n \to \infty} \lambda_{n} = \limweak_{n \to \infty} \bar{\mu}_n = \mu_{\Sigma}
\end{equation*} 
and for every continuous function $f$ we get
 \begin{align*} 
 \lim_{n \to \infty} \frac{1}{N_n}  \sum_{k=0}^n \sigma_{n,k} \sum_{j=1}^{k+1}  f(x_{k+1,j}^{(n-k)}) = \int_{\Rr} f(x) d \mus.
 \end{align*}
\end{theorem}

\begin{proof}
We show that $\nu_{n+1}$ and $\lambda_n$ have the same weak limits. The statement follows then from Theorem \ref{theorem-NevaiweaklimitCD} and Theorem \ref{theorem-Nevaiweaklimit}.
For this, we use explicit bounds for the difference of the measures $K_k^{(n-k)}(x,x) \Dx{\munk}$ and $\sum_{j=1}^{k+1} \delta_{x_{k+1,j}^{(n-k)}}$: according to
\cite[Proposition 2.3]{Simon2009} (see also \cite{Simanek2012WeakCO} for an alternative proof) and the fact that all measures $\munk$ are compactly supported in $[-A-2B,A+2B]$, 
we have for all $l \in \Nn_0$ and $n,k \in \Nn_0$, $n \geq k$, the estimate
\[ \left| \int_{\Rr} x^l K_k^{(n-k)}(x,x) \Dx{\munk} - \sum_{j=1}^{k+1} (x_{k+1,j}^{(n-k)})^l \right| \leq 2 l (A+2B)^l.\]
Since $S$ is a regular summation method, this directly gives
\[ \left| \frac{1}{N_n}  \sum_{k=0}^n \sigma_{n,k} \int_{\Rr} x^l K_k^{(n-k)}(x,x) \Dx{\munk} - \frac{1}{N_n}  \sum_{k=0}^n \sigma_{n,k} \sum_{j=1}^{k+1} (x_{k+1,j}^{(n-k)})^l \right| 
\leq \frac{2 l (A+2B)^l}{N_n}.\]
Since we assume that $N_n$ diverges, $\nu_{n+1} - \lambda_n$ converges weakly to the zero measure for all polynomials and, thus, for all continuous functions on $[-A-2B,A+2B]$. 
The existence of the weak limit $\limweak_{n \to \infty} \lambda_{n} = \mu_{\Sigma}$ is already established by Theorem \ref{theorem-NevaiweaklimitCD} such that
$\limweak_{n \to \infty} \nu_{n+1} = \limweak_{n \to \infty} \lambda_{n} = \mu_{\Sigma}$. 
\end{proof}

By Corollary \ref{cor:uniformnevai}, the equilibrium measure $\mus$ for families of polynomials in the uniform Nevai class $M^{(U)}(a,b)$ is given by the arcsine distribution.  
The preceding theorems now also give us the following result.

\begin{corollary} \label{cor:CDuniformnevai}
Assume that the family $\muk$ is in the uniform Nevai class $M^{(U)}(a,b)$. Then, for any regular N\"orlund method $S$ with $N_n \to \infty$ 
the measures $\lambda_n$ and $\nu_{n+1}$ converge weakly to the equilibrium measure
\[ \Dx{\mus} = \frac{1}{\pi} \frac{\chi_{[a-2b,a+2b]}(x)}{\sqrt{4b^2 - (x-a)^2}}\Dx{x}. \]
\end{corollary}

\section{Weak limits related to ultraspherical polynomials}

\begin{figure}[htb]
 \centering 
 \includegraphics[width= 0.32\textwidth]{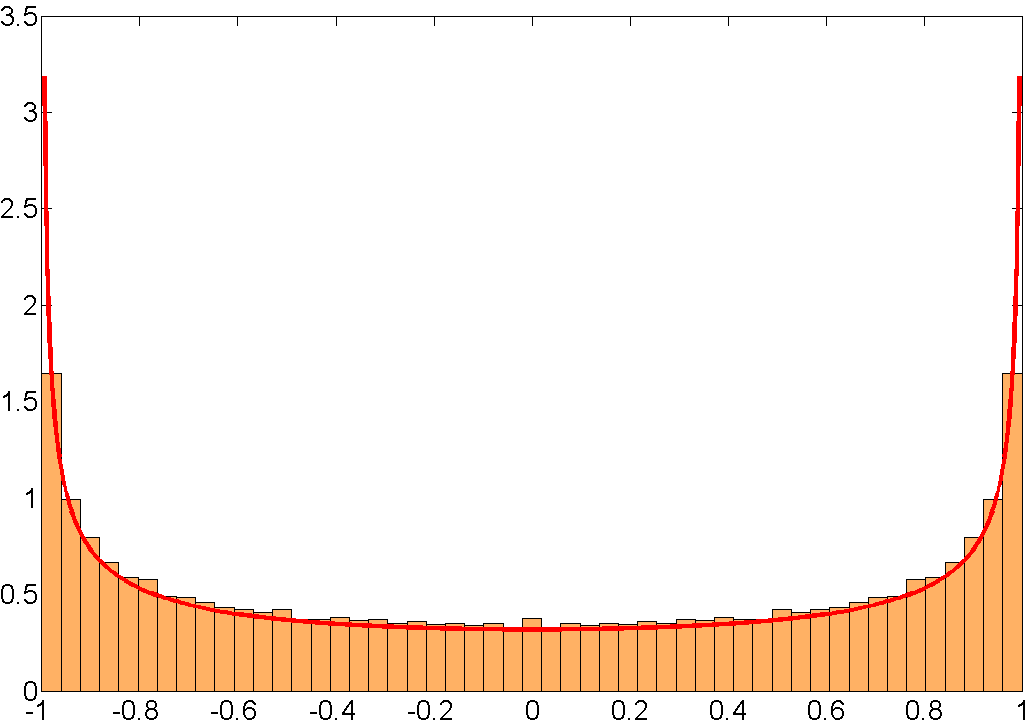} 
 \includegraphics[width= 0.32\textwidth]{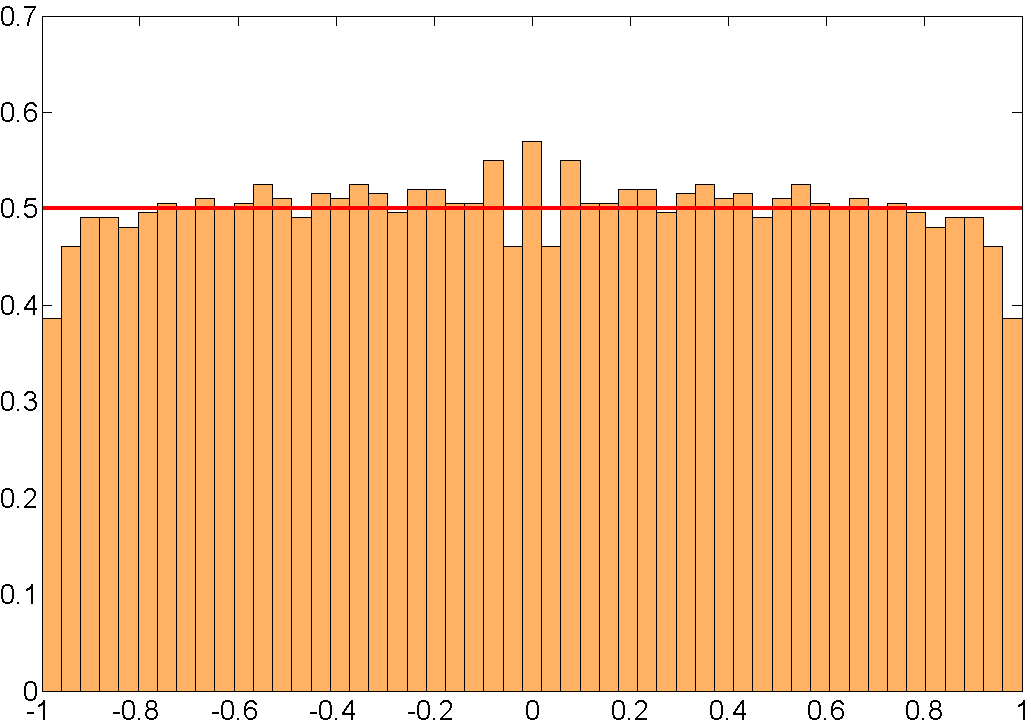}
 \includegraphics[width= 0.32\textwidth]{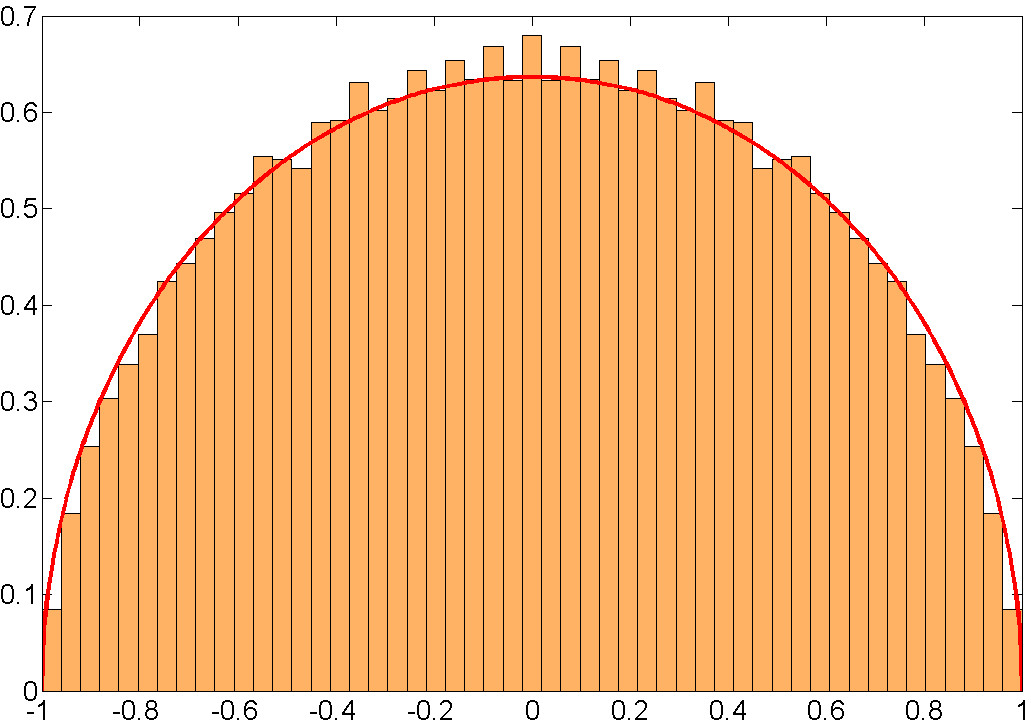} 
 \caption{Graphical illustration of the mean measure $\nu_{n+1}$, $n = 100$, given in definition \eqref{eq:summethod3}. The mean distribution $\nu_{101}$ of the roots
 for families of polynomials in three different Nevai classes is shown orange histograms. Left: the used family $p_n^{(k)} = 
 p_n^{(0,2 \pi +1)}$ is in the uniform Nevai class $M^{(U)}(0,\frac12)$, the equilibrium measure (red) is the arcsine distribution derived in Corollary \ref{cor:CDuniformnevai}. Center: the family $p_n^{(k)} = 
 p_n^{(k,2 \pi +1)}$ is in the mean Nevai class $M^{(C,1)}(\Sigma_{l_a,l_b})$ considered in Corollary \ref{corllary-multivariateweaklimitlegendre}.
 The equilibrium measure (red)
 is the uniform distribution on $[-1,1]$. Right: 
 the ultraspherical family $p_n^{(k)} = p_n^{(k,2 \pi +1)}$ is also in the mean Nevai class $M^{(C,2)}(\Sigma_{l_a,l_b})$ considered in
 Theorem \ref{Theorem-multivariateweaklimitgegenbauerpolynomials}. The corresponding equilibrium measure (red) is $\frac{2}{\pi} \sqrt{1-x^2}\Dx{x}$. }
\end{figure}

As a more concrete example of mean weak limits we consider families of ultraspherical polynomials. 
Using an additional fixed parameter $\lambda > - \frac12$, we consider the family of ultraspherical polynomials 
$p_{n}^{(k,\lambda)}(x)$, $n \in \Nn_0$, $k \in \Nn_0$, orthogonal on $[-1,1]$ with respect to 
the measure $\Dx{\muklambda}(x) = (1-x^2)^{k+\lambda-\frac12} \Dx(x)$. The recurrence coefficients of the polynomials $p_{n}^{(k,\lambda)}(x)$ are explicitly given by (see \cite[p. 29]{Gautschi})
\[a_{k}^{(n-k,\lambda)} = 0, \quad b_{k}^{(n-k,\lambda)} = \frac{1}{2} \sqrt{\frac{k (2 n-k + 2\lambda-1)}{(n + \lambda-1)(n + \lambda)}}.\]
Since, the values $a_{k}^{(n-k,\lambda)} = 0$ are all zero, we have $\Sigma_{l_a,l_b} = 0$ if $l_a \geq 1$. Therefore only the values $\Sigma_{0,l_b}$, $l_b \in 2\Nn$ are relevant
in the determination of the equilibrium measure $\mus$. We list some useful summation methods for the ultraspherical polynomials. \\

{ \noindent \bf 5.1. Arithmetic mean.}
The first summation method is the standard arithmetic mean given by the weights $\sigma_{n,k}^{(C,1)} = \frac1{n+1}$. 
The arithmetic mean is a regular N\"orlund method and can be decomposed as $\sigma_{n,k}^{(C,1)} = \sigone_n^{(C,1)} \sigtwo_{n-k}^{(C,1)}$ with $\sigone_n^{(C,1)} = \frac1{n+1}$ 
and $\sigtwo_{n-k}^{(C,1)}=1$. 
The normalizing constant $N_n$ is given by
\[N_n^{(C,1)} = \frac{1}{n+1} \sum_{k=0}^n (k+1) = \frac{n+2}{2}.\]
By using the definition of the Riemann integral and the explicit formulas for $a_{k}^{(n-k,\lambda)}$, $b_{k}^{(n-k,\lambda)}$, 
we can compute the values $\Sigma_{0,l_b}$, $l_b \in 2\Nn$, for the family $\muklambda$:
\begin{align*} 
\Sigma_{0,l_b} &= \lim_{n\to \infty} 2^{-l_b} \sum_{k=0}^{n} \frac1{n+1} \frac{k^{\frac{l_b}{2}} (2 n-k + 2\lambda-1)^{\frac{l_b}{2}}}{(n + \lambda-1)^{\frac{l_b}{2}}(n + \lambda)^{\frac{l_b}{2}}} \\
               &= 2^{-l_b} \lim_{n\to \infty} \left( \frac{n^2}{(n + \lambda-1)(n+\lambda)}\right)^{\frac{l_b}{2}} \lim_{n\to \infty} \sum_{k=0}^{n} 
               \frac1{n+1} \left( \frac{k}{n}\right)^{\frac{l_b}{2}} \left(2 -\frac{k}{n} + \frac{2\lambda-1}{n}\right)^{\frac{l_b}{2}} \\ 
               &= 2^{-l_b} \lim_{n\to \infty} \sum_{k=1}^{n} 
               \frac1n \left( \frac{k}{n}\right)^{\frac{l_b}{2}} \left(2 -\frac{k}{n} \right)^{\frac{l_b}{2}} \\
               &= 2^{-l_b} \int_0^1 (x(2-x))^{\frac{l_b}{2}} \Dx{x} = 2^{-l_b-1} \int_{-1}^1 (1-x^2)^{\frac{l_b}{2}} \Dx{x} = 
               \frac{\Gamma(\frac{3}{2})}{2^{l_b}}\frac{\Gamma(\frac{l_b+2}{2})}{\Gamma(\frac{l_b+3}{2})}.
\end{align*}
The final identity is a standard formulation in terms of the Gamma function. \\

{ \noindent \bf 5.2. Legendre summation.}
As a second summation method, we consider a summation method related to the addition formula for the Legendre polynomials. The 
weights for this summation method are defined by $\sigma_{n,k}^{(L)} = \sigone_n^{(L)} \sigtwo_{n-k}^{(L)}$ with 
$$\sigone_n^{(L)} = \frac{1}{2n+1} \quad \text{and} \quad \sigtwo_k^{(L)} = \left\{ \begin{array}{ll}
                                                                         1, & k=0 \\
                                                                         2, & k > 0.
                                                                         \end{array}
\right. $$
Exactly as the arithmetic mean, also this summation is a regular N\"orlund method. For the normalizing constant $N_n^{(L)}$ we get
\[N_n^{(L)} = \frac{1}{2n+1} \left( \sum_{k=0}^n ( 2 k+1) \right) = \frac{(n+1)^2}{2n+1}.\]
In the same way as for the arithmetic mean, we get the following limits $\Sigma_{0,l_b}$ for $l_b \in 2\Nn$:
\begin{align*} 
\Sigma_{0,l_b} &= \lim_{n\to \infty} 2^{-l_b} \sum_{k=1}^{n} \frac{2}{2n+1} \frac{k^{\frac{l_b}{2}} (2 n-k + 2\lambda-1)^{\frac{l_b}{2}}}{(n + \lambda-1)^{\frac{l_b}{2}}(n + \lambda )^{\frac{l_b}{2}}} 
 = \frac{\Gamma(\frac{3}{2})}{2^{l_b}}\frac{\Gamma(\frac{l_b+2}{2})}{\Gamma(\frac{l_b+3}{2})}.
\end{align*}

{ \noindent \bf 5.3. Ces\`{a}ro $(C,\alpha)$ summation.} A generalization of the arithmetic mean is Ces\`{a}ro $(C,\alpha)$ summation. For $\alpha > 0$, this summation method is regular with 
the weights $\sigma_{n,k}^{(C,\alpha)} = \sigone_n^{(C,\alpha)} \sigtwo_{n-k}^{(C,\alpha)}$ given by
$$\sigone_n^{(C,\alpha)} = \frac{1}{\binom{n+\alpha}{n}} \quad \text{and} \quad \sigtwo_k^{(C,\alpha)} = \binom{k+\alpha-1}{k}. $$  
For the normalizing constant $N_n^{(C,\alpha)}$ we get
\[N_n^{(C,\alpha)} = \frac{\binom{n+\alpha+1}{n}}{\binom{n+\alpha}{n}} = \frac{n + \alpha  + 1}{\alpha+1}.\]
We calculate now the limits $\Sigma_{0,l_b}$, $l_b \in 2\Nn$, for the family of ultraspherical measures $\Dx{\muklambda}(x) = (1-x^2)^{k+\lambda-\frac12} \Dx(x)$:
\begin{align*} 
\Sigma_{0,l_b} &= \lim_{n\to \infty} 2^{-l_b} \sum_{k=0}^{n} \frac{\binom{n-k+\alpha-1}{n-k}}{\binom{n+\alpha}{n}} 
               \frac{k^{\frac{l_b}{2}} (2 n-k + 2\lambda-1)^{\frac{l_b}{2}}}{(n + \lambda-1)^{\frac{l_b}{2}}(n + \lambda)^{\frac{l_b}{2}}} \\          
               &= 2^{-l_b} \frac{\Gamma(\alpha+1)}{\Gamma(\alpha)}\lim_{n\to \infty} \sum_{k=0}^{n} 
               \frac1n \left( 1 - \frac{k}{n}\right)^{\alpha-1} \left( \frac{k}{n}\right)^{\frac{l_b}{2}} \left(2 -\frac{k}{n} \right)^{\frac{l_b}{2}} \\
               &= 2^{-l_b} \frac{\Gamma(\alpha+1)}{\Gamma(\alpha)} \int_0^1 (1-x)^{\alpha-1}(x(2-x))^{\frac{l_b}{2}} \Dx{x} = 2^{-l_b} \frac{\alpha}{2}\int_{-1}^1 |x|^{\alpha-1}(1-x^2)^{\frac{l_b}{2}} \Dx{x} \\ 
               &= \frac{\Gamma(\frac{\alpha+2}{2})}{2^{l_b}}\frac{\Gamma(\frac{l_b+2}{2})}{\Gamma(\frac{\alpha+l_b+2}{2})} 
               = 2^{-l_b} \binom{\frac{\alpha+l_b}{2}}{\frac{\alpha}{2}}^{-1}.
\end{align*}

{ \noindent \bf 5.4. Gegenbauer summation.}
A direct generalization of the Legendre mean is the Gegenbauer summation $(G,\nu)$ for $\nu > 0$. 
The weights $\sigma_{n,k}^{(G,\nu)} = \sigone_n^{(G,\nu)} \sigtwo_{n-k}^{(G,\nu)}$ are in this case given by
$$\sigone_n^{(G,\nu)} = \frac{\Gamma(2 \nu + 1)\Gamma(n+1)}{(2n+2\nu) \Gamma(n+2\nu) } \quad \text{and} 
\quad \sigtwo_k^{(G,\nu)} = \frac{(2k+2\nu-1) \Gamma(k+2\nu-1) }{\Gamma(2 \nu )\Gamma(k+1)}. $$  
This formula is well-defined for all $n,k \in \Nn_0$ and $\nu > 0$, $\nu \neq \frac12$. In the case $\nu = \frac12$ we define the weights by 
taking the limit $ \nu \to \frac12$. In this way, we obtain $\sigtwo_k^{(G,1/2)} = \sigtwo_k^{(L)}$, i.e. the Gegenbauer summation for 
$\nu = \frac12$ corresponds to the Legendre summation discussed before. 
For the normalizing constant $N_n^{(G,\nu)}$ we have
\[N_n^{(G,\nu)} = \frac{(2n+2\nu+1) (n+2\nu) }{(2n+2\nu)(2 \nu + 1 )}.\]
The identities for $\sigone_n^{(G,\nu)}$ and $N_n^{(G,\nu)}$ can be deduced easily from a relation of the Gegenbauer summation
to the Ces\`{a}ro means. Namely, for the sequences $\sigtwo_k^{(G,\nu)}$ we have the relations
\begin{equation} \label{eq:relcesarogegenbauer}
\sigtwo_k^{(G,\nu)} = \left\{ \begin{array}{ll}
                                                                         \sigtwo_k^{(C,2\nu)}, & k=0, \\
                                                                         \sigtwo_k^{(C,2\nu)}+\sigtwo_{k-1}^{(C,2\nu)}, & k > 0.
                                                                         \end{array}
\right.
\end{equation}
This immediately implies also for the sequence $\sigone_n$ the relation
\begin{equation} \label{eq:relcesarogegenbauer2}
\sigone_n^{(G,\nu)} = \left\{ \begin{array}{ll}
                                                                         \sigone_n^{(C,2\nu)}, & n=0, \\
                                                                         \left(\frac{1}{\sigone_n^{(C,2\nu)}}+\frac{1}{\sigtwo_{n-1}^{(C,2\nu)}}\right)^{-1}, & n > 0.
                                                                         \end{array}
\right. 
\end{equation}
Using the relations \eqref{eq:relcesarogegenbauer} and \eqref{eq:relcesarogegenbauer2} gives the above stated explicit formulas for $\sigone_n^{(G,\nu)}$ and $N_n^{(G,\nu)}$, and,
even more, we obtain the limits $\Sigma_{0,l_b}$ for the family of ultraspherical measures from the corresponding limits of the Ces\`{a}ro means:
\begin{align*} 
\Sigma_{0,l_b} = 2^{-l_b} \binom{\nu + \frac{l_b}{2}}{\nu}^{-1}, \qquad \text{for $\nu > 0$ and $l_b \in 2 \Nn$.}
\end{align*}
Note that in the case $\nu = \frac12$ we obtain precisely the derived formula for the Legendre summation. There is a relation of
the weights $\sigma_{n,k}^{(G,\lambda)}$, $\lambda > 0$, with the addition formula of the ultraspherical polynomials $p_k^{(n-k,\lambda)}$. 
Applying the general addition formula \cite[3.15.1, (19)]{Erdelyi} or \cite[(18.18.8)]{OlverNist} to the orthonormal ultraspherical polynomials
$p_k^{(n-k,\lambda)}$, we obtain the following special variant of the addition formula:
\begin{equation} \label{eq:additionultraspherical}
\frac{1}{m^{(\lambda)}} = \sum_{k=0}^n \sigma_{n,k}^{(G,\lambda)} (1-x^2)^{n-k} |p_{k}^{(n-k,\lambda)}(x)|^2
\end{equation}
where $m^{(\lambda)} = \int_{-1}^{1} (1-x^2)^{\lambda-\frac{1}{2}}\Dx{x} = \pi 2^{1-2 \lambda} \binom{2\lambda-1}{\lambda}$ and $x \in [-1,1]$. In the Legendre case $\lambda = \frac{1}{2}$, 
a simplified version of the addition formula can be obtained from \cite[(18.18.9)]{OlverNist}:
\begin{equation} \label{eq:additionlegendre}
\frac{2n+1}{2} = |p_{n}^{(0,\frac12)}(x)|^2 + 2 \sum_{k=1}^{n} (1-x^2)^{k} |p_{n-k}^{(k,\frac12)}(x)|^2.
\end{equation}
The formula \eqref{eq:additionlegendre} can alternatively be regarded as a version of the addition theorem for spherical harmonics. The two given formulas are also a special case
of a more general addition formula for Jacobi polynomials, see \cite{Koornwinder1972}. Both formulas turn out to be very useful when deriving explicit identities for 
the equilibrium measure in the ultraspherical case.\\

{\noindent \bf 5.5. Weak limits related to Ces\`{a}ro means of ultraspherical polynomials.}

\begin{theorem} \label{Theorem-multivariateweaklimitgegenbauerpolynomials}
Suppose that the family $\muk$, $k \in \Nn_0$ is in the mean Nevai class $M^{(C,\alpha)}(\Sigma_{l_a,l_b})$ with $\alpha > 0$ and limits $\Sigma_{l_a,l_b}$ given by 
\begin{equation} \label{eq:sigmalimits}
\Sigma_{l_a,l_b} =  \delta_{0,l_a} \, 2^{-l_b}\binom{\frac{\alpha+l_b}{2}}{\frac{\alpha}{2}}^{-1}.                                                                      
\end{equation}
Then, the equilibrium measure $\mus$ is given as
\[ \Dx{\mus} = \frac{1}{m^{(\frac{\alpha}{2})}} \chi_{[-1,1]}(x) (1 - x^2)^{\frac{\alpha-1}{2}}\Dx{x}. \]
with the normalization $m^{(\frac{\alpha}{2})} = \int_{-1}^{1} (1-x^2)^{\frac{\alpha-1}{2}}\Dx{x} = \pi 2^{1-\alpha} \binom{\alpha-1}{\frac{\alpha}{2}}$.
In particular, if $f$ is a continuous function on $[-1,1]$, we have the following weak limits:
\begin{align} \label{equation-weaklimitgegenbauerpolynomials1}
& \lim_{n\to \infty} \sum_{k=0}^n \frac{\binom{n-k+\alpha-1}{n-k}}{\binom{n+\alpha}{n}} \int_{-1}^{1} f(x) |p_{k}^{(n-k)}(x)|^2 \Dx{\munk}(x) 
= \frac{1}{m^{(\frac{\alpha}{2})}} \int_{-1}^{1} f(x) (1-x^2)^{\frac{\alpha-1}{2}}\Dx{x},\\
& \lim_{n\to \infty} \sum_{k=0}^n \frac{\binom{n-k+\alpha-1}{n-k}}{\binom{n+\alpha+1}{n}} \int_{-1}^{1} f(x) K_{k}^{(n-k)}(x,x) \Dx{\munk}(x)
= \frac{1}{m^{(\frac{\alpha}{2})}} \int_{-1}^{1} f(x) (1-x^2)^{\frac{\alpha-1}{2}}\Dx{x}, \label{equation-weaklimitgegenbauerpolynomials2} \\
& \lim_{n\to \infty} \sum_{k=0}^n \frac{\binom{n-k+\alpha-1}{n-k}}{\binom{n+\alpha+1}{n}} \sum_{j=1}^{k+1} f(x_{k+1,j}^{(n-k)})
= \frac{1}{m^{(\frac{\alpha}{2})}} \int_{-1}^{1} f(x) (1-x^2)^{\frac{\alpha-1}{2}}\Dx{x}. \label{equation-weaklimitgegenbauerpolynomials3}
\end{align}
\end{theorem}

\begin{proof}
By the description given in Section 5.3, the Ces\`{a}ro mean $(C,\alpha)$ is a regular N\"orlund summation method. We further know that $N_n^{(C,\alpha)} \to \infty$ as $n \to \infty$. 
Then, by the Theorems \ref{theorem-NevaiweaklimitCD}, \ref{theorem-Nevaiweaklimitpointmeasure} and \ref{theorem-Nevaiweaklimit}, we have an equilibrium measure $\mus$ determined by 
the numbers $\Sigma_{l_a,l_b}$ given in \eqref{eq:sigmalimits} such that
\[ \limweak_{n \to \infty} \nu_{n+1} = \limweak_{n \to \infty} \lambda_{n} = \limweak_{n \to \infty} \bar{\mu}_n = \mus.\]
In particular the limits in \eqref{equation-weaklimitgegenbauerpolynomials1}, \eqref{equation-weaklimitgegenbauerpolynomials2} and \eqref{equation-weaklimitgegenbauerpolynomials3} are 
identical and it only remains to determine the explicit form of the equilibrium measure $\mus$. To obtain this formula, we consider the slightly modified Gegenbauer summation 
method given in Section 5.4. For the family of measures $\Dx{\muklambda}(x) = (1-x^2)^{k+\lambda-\frac12} \Dx(x)$, $k \in \Nn_0$, $\lambda > 0$, The Gegenbauer summation 
method $(G,\frac{\alpha}{2})$ gives the same values $\Sigma_{l_a,l_b}$ as the Ces\`{a}ro mean $(C,\alpha)$ and, therefore by Theorem \ref{theorem-Nevaiweaklimit}, 
also the weak equilibrium measures $\mus$ are identical. Using in addition the addition formula \eqref{eq:additionultraspherical} related to the ultraspherical polynomials
$p_{k}^{(n-k,\frac{\alpha}{2})}(x)$ we obtain the following identities:
\begin{align*} 
& \hspace{-1cm} \lim_{n\to \infty} \sum_{k=0}^n \frac{\binom{n-k+\alpha-1}{n-k}}{\binom{n+\alpha}{n}} \int_{-1}^{1} f(x) |p_{k}^{(n-k)}(x)|^2 \Dx{\munk}(x) \\
& = \lim_{n\to \infty} \sum_{k=0}^n \sigma_{n,k}^{(C,\alpha)} \int_{-1}^{1} f(x) |p_{k}^{(n-k,\frac{\alpha}{2})}(x)|^2 \Dx{\mu^{(n-k,\frac{\alpha}{2})}}(x) \\
& = \lim_{n\to \infty} \int_{-1}^{1} f(x) \sum_{k=0}^n \sigma_{n,k}^{(G,\frac{\alpha}{2})} |p_{k}^{(n-k,\frac{\alpha}{2})}(x)|^2 \Dx{\mu^{(n-k,\frac{\alpha}{2})}}(x) \\
& = \lim_{n\to \infty} \int_{-1}^{1} f(x) \frac{1}{m^{(\frac{\alpha}{2})}} (1 - x^2)^{\frac{\alpha-1}{2}}\Dx{x} =  \frac{1}{m^{(\frac{\alpha}{2})}} 
\int_{-1}^{1} f(x)  (1 - x^2)^{\frac{\alpha-1}{2}}\Dx{x}.
\end{align*}
For $\alpha = 1$ we used the respective addition formula \eqref{eq:additionlegendre} of the Legendre case. 
\end{proof}

\begin{remark}
From the details in the proof of Theorem \ref{Theorem-multivariateweaklimitgegenbauerpolynomials} we see that the statements of Theorem \ref{Theorem-multivariateweaklimitgegenbauerpolynomials}
hold also true if we replace the Ces\`{a}ro summation method $(C,\alpha)$ with the Gegenbauer summation method $(G,\frac{\alpha}{2})$. 
The limit identity \eqref{equation-weaklimitgegenbauerpolynomials1} can be regarded as an asymptotic weak addition formula for the class of measures $M^{(C,\alpha)}(\Sigma_{l_a,l_b})$. It
therefore makes sense to denote the class $M^{(C,\alpha)}(\Sigma_{l_a,l_b})$ (or also $M^{(G,\frac{\alpha}{2})}(\Sigma_{l_a,l_b})$) given in 
Theorem \ref{Theorem-multivariateweaklimitgegenbauerpolynomials} as ultraspherical mean Nevai class. 
\end{remark}

We consider some particular examples and extensions of Theorem \ref{Theorem-multivariateweaklimitgegenbauerpolynomials}. 

\begin{example}
\begin{itemize}
 \item The family $\Dx{\muk} = (1-x^2)^k (1-x)^{\lambda_1}(1+x)^{\lambda_2} \Dx{x}$ with fixed parameters $\lambda_1,\lambda_2 > -1$ is in the Nevai class $M^{(C,\alpha)}(\Sigma_{l_a,l_b})$
 with the values $\Sigma_{l_a,l_b}$ given in \eqref{eq:sigmalimits}. As for the ultraspherical polynomials, this can be checked directly using the three-term recurrence relation
 of the Jacobi polynomials (given, for example, in \cite[p. 29]{Gautschi}). Thus, Theorem \ref{Theorem-multivariateweaklimitgegenbauerpolynomials} yields the weak limits
 \eqref{equation-weaklimitgegenbauerpolynomials1}, \eqref{equation-weaklimitgegenbauerpolynomials2} and \eqref{equation-weaklimitgegenbauerpolynomials3} for this family of measures. 
 \item We consider the family $\Dx{\muk} = (1-x^2)^{k+\lambda-\frac{1}{2}} e^{(2\theta - \pi)t}|\Gamma(\lambda+k+i t)|^2 \Dx{x} $ of Pollaczek measures with 
 $x = \cos \theta$, $t = \frac{a x + b}{\sqrt{1-x^2}}$ and fixed parameters $\lambda>0$, $a \geq |b| $, see \cite[Chapter VI, 5]{Chihara}. This family of Pollaczek measures also 
 satisfies the conditions of Theorem \ref{Theorem-multivariateweaklimitgegenbauerpolynomials}. The values $\Sigma_{l_a,l_b}$ can be calculated directly from the 
 three-term recurrence coefficients of the Pollaczek polynomials in the same way as the calculations were carried out in Section 5.3 for the ultraspherical polynomials. In fact,
 for $a = b = 0$ the Pollaczek polynomials correspond to the ultraspherical polynomials. 
\end{itemize}
\end{example}

In the special case $\alpha = 1$ we get the following weak limit for the arithmetic mean. It can be regarded as an extended variant of the limit relations derived 
for co-recursive ultraspherical polynomials in \cite{ErbMathias2015}.

\begin{corollary} \label{corllary-multivariateweaklimitlegendre}
Let the family $\muk$ be in the mean Nevai class $M^{(C,1)}(\Sigma_{l_a,l_b})$ with the values $\Sigma_{l_a,l_b}$ given by 
\begin{equation} \label{eq:sigmalimitslegendre}
\Sigma_{l_a,l_b} =  \delta_{0,l_a} \, \frac{\Gamma(\frac{3}{2})}{2^{l_b}}\frac{\Gamma(\frac{l_b+2}{2})}{\Gamma(\frac{l_b+3}{2})}.                                                                    
\end{equation}
Then, the equilibrium measure $\mus$ is the uniform probability measure on $[-1,1]$ and for every continuous function $f$ on $[-1,1]$ we have
\begin{equation} \label{equation-weaklimitlegenrezeros}
\lim_{n\to \infty} \sum_{k=0}^n \frac{1}{n \!+\!1} \! \int_{-1}^{1} f(x) |p_{k}^{(n-k)}(x)|^2 \Dx{\munk}\!(x) = 
\lim_{n\to \infty}  \sum_{k=0}^{n} \sum_{j=1}^{k+1} \frac{2 f(x_{k+1,j}^{(n-k)}) }{(n+1)(n+2)}  = \frac{1}{2} \! \int_{-1}^1 f(x) \Dx{x}.
\end{equation}
\end{corollary}

The linear function $T: \Rr \to \Rr$, $T y = 2b y + a$ maps the interval $[-1,1]$ onto $[a-2b,a+2b]$. Using this linear map, we can transfer 
Theorem \ref{Theorem-multivariateweaklimitgegenbauerpolynomials} easily to an arbitrary interval $[a-2b,a+2b]$.

\begin{corollary} \label{corllary-multivariateweaklimitgeneral}
Suppose that the family $\muk$ is in the mean Nevai class $M^{(C,\alpha)}(\Sigma_{l_a,l_b})$ with $\alpha > 0$ and limits $\Sigma_{l_a,l_b}$ given by 
\begin{equation} \label{eq:sigmalimitsgeneral}
\Sigma_{l_a,l_b} =  a^{l_a} \, b^{l_b} \binom{\frac{\alpha+l_b}{2}}{\frac{\alpha}{2}}^{-1}.                                                                      
\end{equation}
Then, the equilibrium measure $\mus$ is given by
\[ \Dx{\mus} = \frac{1}{2 \pi b^{\alpha} \binom{\alpha-1}{\frac{\alpha}{2}}} \chi_{[a - 2b, a + 2b]}(x) (b^2 - (x-a)^2)^{\frac{\alpha-1}{2}} \Dx{x}. \]
\end{corollary}


\end{document}